\pdfoutput=1
\RequirePackage{ifpdf}
\ifpdf 
\documentclass[pdftex]{sigma}
\else
\documentclass{sigma}
\fi

\numberwithin{equation}{section}

\newtheorem{Theorem}{Theorem}[section]
\newtheorem*{Theorem*}{Theorem}

\newtheorem{Proposition}[Theorem]{Proposition}
 { \theoremstyle{definition}
\newtheorem{Definition}[Theorem]{Definition}

\newtheorem{Example}[Theorem]{Example}
\newtheorem{Remark}[Theorem]{Remark} }

\def\C{{\mathbb C}}
\def\N{{\mathbb N}}
\def\R{{\mathbb R}}

\def\sgn{\operatorname{sgn}}
\def\dd{ \mathrm{d}}
\def\Det{\operatorname{Det}}
\def\tr{\operatorname{tr} }
\def\CC{\mathcal{C}}
\def\mcalD{\mathcal{D}}
\def\mcalHn{\mathcal{H}^{n \times n}}
\def\mcalM{\mathcal{M}}
\def\frS{\mathfrak{S}}
\def\uu{\boldsymbol{u}}
\def\vv{\boldsymbol{v}}
\def\ww{\boldsymbol{w}}
\def\xx{\boldsymbol{x}}

\begin{document}

\newcommand{\arXivNumber}{2412.03000}

\renewcommand{\thefootnote}{}

\renewcommand{\PaperNumber}{055}

\FirstPageHeading

\ShortArticleName{Hyperdeterminantal Total Positivity}

\ArticleName{Hyperdeterminantal Total Positivity\footnote{This paper is a~contribution to the Special Issue on Basic Hypergeometric Series Associated with Root Systems and Applications in honor of Stephen C.~Milne's 75th birthday. The~full collection is available at \href{https://www.emis.de/journals/SIGMA/Milne.html}{https://www.emis.de/journals/SIGMA/Milne.html}}}

\Author{Kenneth W.~JOHNSON~$^{\rm a}$ and Donald St.~P.~RICHARDS~$^{\rm b}$}

\AuthorNameForHeading{K.W.~Johnson and D.St.P.~Richards}

\Address{$^{\rm a)}$~Department of Mathematics, Pennsylvania State University, \\
\hphantom{$^{\rm a)}$}~Abington, Pennsylvania 19001, USA}
\EmailD{\href{mailto:kwj1@psu.edu}{kwj1@psu.edu}}

\Address{$^{\rm b)}$~Department of Statistics, Pennsylvania State University, \\
\hphantom{$^{\rm b)}$}~University Park, PA 16802, USA}
\EmailD{\href{mailto:dsr11@psu.edu}{dsr11@psu.edu}}

\ArticleDates{Received December 04, 2024, in final form July 03, 2025; Published online July 11, 2025}

\Abstract{For a given positive integer $m$, the concept of {\it hyperdeterminantal total positivity} is defined for a kernel $K\colon {\mathbb R}^{2m} \to {\mathbb R}$, thereby generalizing the classical concept of total positivity. Extending the fundamental example, $K(x,y) = \exp(xy)$, $x, y \in \mathbb{R}$, of a classical totally positive kernel, the hyperdeterminantal total positivity property of the kernel $K(x_1,\dots,x_{2m}) = \exp(x_1\cdots x_{2m})$, $x_1,\dots,x_{2m} \in \mathbb{R}$ is established. By applying Matsumoto's hyperdeterminantal Binet--Cauchy formula, we derive a generalization of Karlin's basic composition formula; then we use the generalized composition formula to construct several examples of hyperdeterminantal totally positive kernels. Further generalizations of hyperdeterminantal total positivity by means of the theory of finite reflection groups are described and some open problems are posed.}


\Classification{33C20; 05E05; 15A15; 15A72; 33C80}

\begin{flushright}
\begin{minipage}{80mm}\it
This article is dedicated to Steve Milne, who has displayed over the years a~masterful use of the Schur functions in algebraic and analytic settings. We heartily congratulate Steve on the occasion of his 75th birthday and wish him many more years of beautiful applications of the Schur functions and their properties.
\end{minipage}
\end{flushright}

\renewcommand{\thefootnote}{\arabic{footnote}}
\setcounter{footnote}{0}

\section{Introduction}\label{sec_intro}

The theory of total positivity has, for the past 95 years, played an increasingly important role in many aspects of the mathematical sciences. The theory was largely initiated by the work of Schoenberg \cite{Schoenberg} and Krein \cite{Krein} in the 1930's, P\'olya \cite{Polya} in the 1940's, and Gantmacher and Krein \cite{GantmacherKrein} and Karlin \cite{Karlin} in the 1950's, and it is remarkable that each of those early authors was motivated by a wide variety of considerations such as approximation theory, frequency sequences and functions, integral operators, mathematical economics, ordinary differential equations and related Green's functions, reliability theory, splines, statistics, and variation-diminishing transformations.

Since the appearance of Karlin's monograph \cite{Karlin}, the theory of total positivity has broadened immensely, to the point where it arises in combinatorics, classical special functions, Gr\"obner bases, matrix analysis, Lie groups, harmonic analysis, special functions of matrix argument, mathematical physics, operations research, partial differential equations, stochastic processes, actuarial mathematics, conformal field theory (high energy physics), statistical mechanics, evolutionary biology, computer-aided geometric design, and many other areas. We refer to Gantmacher and Krein \cite{GantmacherKrein}, Gasca and Micchelli \cite{GascaMicchelli}, Karlin \cite{Karlin,Karlin2}, Pinkus \cite{Pinkus}, and Gelfand, Kapranov, and Zelevinsky \cite{GKZ} for various accounts and applications of the classical theory and numerous references to the larger literature.

Let $\mcalD \subseteq \R^2$ and $d \in \N$. In the classical setting, a kernel $K\colon \mcalD \to \R$ is {\it totally positive of order}~$d$, denoted TP$_d$, if the $n \times n$ determinant
\begin{equation}
\det(K(x_j,y_k)) =
\begin{vmatrix}
K(x_1,y_1) & K(x_1,y_2) & \cdots & K(x_1,y_n) \\
K(x_2,y_1) & K(x_2,y_2) & \cdots & K(x_2,y_n) \\
\vdots & \vdots & & \vdots \\
K(x_n,y_1) & K(x_n,y_2) & \cdots & K(x_n,y_n)
\end{vmatrix}\label{eq_det_K}
\end{equation}
is nonnegative for all $n=1,\dots,d$, and for all $x_1 > \cdots > x_d$ and $y_1 > \cdots > y_d$ such that~${(x_j,y_k) \in \mcalD}$ for all $j,k=1,\dots,d$. Equivalently, for a kernel $K \colon \mcalD \to [0,\infty)$, total positivity of order $d$ means that all minors of the $d \times d$ matrix $(K(x_i,y_j))$ are nonnegative whenever~${x_1 > \cdots > x_d}$ and $y_1 > \cdots > y_d$, i.e., all $n \times n$ submatrices of the matrix $(K(x_i,y_j))$ have nonnegative determinant, for all $n=1,\dots,d$.

Consider the kernel
\begin{equation}
K(x,y) = \exp(xy), \qquad (x,y) \in \R^2,\label{eq_exp_tp}
\end{equation}
a function that is the fundamental example of a totally positive kernel. This kernel will provide a guiding role in our subsequent results. It is well known that this kernel is totally positive of order $d$ for all $d \in \N$, a property that is denoted by TP$_\infty$. Moreover, the minor \eqref{eq_det_K} is positive for all $n \in \N$, so we say that $K$ is {\it strictly totally positive of order infinity}, denoted STP$_\infty$.

Returning to the general definition in \eqref{eq_det_K}, it may be observed that the classical concept of total positivity is defined (only) for kernels defined on $\R^2$ and for all orders $d \in \N$. On the other hand, for kernels $K \colon \R^m \to \R$, where $m \ge 2$, the concept of total positivity is defined only for orders $d \le 2$, with $d=1$ signifying that $K$ is nonnegative; specifically, the concept of {\it multivariate totally positivity of order $2$} for kernels $K \colon \R^m \to [0,\infty)$ was defined in the article~\cite{karlinrinott80}. Thus no general treatment has yet been developed for the concept of multivariate total positivity of order $d$, for any $d > 2$, when the kernel $K$ is defined on $\R^m$, for arbitrary $m > 2$.

In this article, for a given positive integer $m$, we follow \cite{Johnson_Richards,wanglixu} by applying the theory of Cayley's first hyperdeterminant to define the concept of {\it hyperdeterminantal total positivity} (HTP) for a kernel $K \colon \R^{2m} \to \R$. For the case in which $m=1$, this definition reduces to the classical concept of total positivity which was defined through the minors in \eqref{eq_det_K}. Generalizing the fundamental example in \eqref{eq_exp_tp} of a classical totally positive kernel, we shall establish that the kernel
\begin{equation}
K(x_1,\dots,x_{2m}) = \exp(x_1\cdots x_{2m}), \qquad (x_1,\dots,x_{2m}) \in \R^{2m},\label{eq_exp_htp}
\end{equation}
is a HTP kernel of arbitrarily large order. We show, moreover, that the resulting hyperdeterminants are strictly positive, which generalizes the STP$_\infty$ property of the fundamental example~\eqref{eq_exp_tp}.

We remark that the articles \cite{wanglixu} derived results for the total positivity of Hankel and Vandermonde arrays that are equivalent to the hyperdeterminantal properties of those kernels. However, to the best of our knowledge, our results in the present article appear to be the first to connect the theory of HTP arrays with the generalized hypergeometric functions of matrix argument.

In Section \ref{sec_hyperdeterminants}, we provide the definition of the hyperdeterminant and several of its properties. We provide in Section \ref{sec_BC_for_hyperdets} a hyperdeterminantal Binet--Cauchy formula, due to Matsumoto \cite{matsumoto}, and then we apply that result to obtain a Schur function summation formula for the hyperdeterminant constructed from the kernel \eqref{eq_exp_htp}.

In Section \ref{sec_extension_HC}, we extend to the hyperdeterminantal setting an integral of Harish-Chandra \cite{harishchandra} that now appears prominently in the theory of total positivity \cite{grossrichards89,grossrichards95}. Further, we apply the extended integral to derive Schur function summation formulas for certain hyperdeterminants defined in terms of the classical generalized hypergeometric series. In Section \ref{sec_HTP}, we provide the definition of hyperdeterminantal total positivity, obtain several examples of such kernels, and generalize the classical basic composition formula \cite[p.~17]{Karlin}. As consequences of the hyperdeterminantal Binet--Cauchy formula and the generalized basic composition formula, we demonstrate how numerous examples of HTP kernels can be constructed.

Finally, in Section \ref{sec_conclusions}, we describe some directions for future research that are opened by this article. In particular, we raise the possibility of investigating further generalizations of HTP kernels by means of the theory of finite reflection groups, and we raise the problem of deriving hyperdeterminantal generalizations of the FKG inequality.

\section{Hyperdeterminants}
\label{sec_hyperdeterminants}

For any collection of indices $r_1,\dots,r_l \in \{1,\dots,n\}$, let $A(r_1,\dots,r_l) \in \C$ and form the multidimensional \textit{array}, or \textit{tensor},
$A = (A(r_1,\dots,r_l))_{1 \le r_1,\dots,r_l \le n}$. Also let $\frS_n$ denote the symmetric group on $n$ symbols. Cayley's first definition in \cite{cayley46,cayley89, cayley45} of the {\it hyperdeterminant} of the array~$A$~is
\[
\Det\bigl(A(r_1,\dots,r_l)\bigr)_{1 \le r_1,\dots,r_l \le n}
= \frac{1}{n!} \sum_{\sigma_1 \in \frS_n} \cdots \sum_{\sigma_l \in \frS_n} \left(\prod_{k=1}^l \sgn(\sigma_k)\right) \cdot \prod_{j=1}^n A(\sigma_1(j),\dots,\sigma_l(j)).
\]
By replacing each $\sigma_r$ in the multiple summation by $\sigma_0 \sigma_r$, where $\sigma_0$ is an odd permutation, we find that $\Det\bigl(A(r_1,\dots,r_l)\bigr) = (-1)^l \Det\bigl(A(r_1,\dots,r_l)\bigr)$. Consequently, $\Det\bigl(A(r_1,\dots,r_l)\bigr) \equiv 0$ if~$l$ is odd, a well known result; see, e.g., \cite[Section 2]{luquethibon}. Therefore, we assume henceforth that $l$ is even, $l = 2m$, so that
\begin{gather}
\Det(A(r_1,\dots,r_{2m}))_{1 \le r_1,\dots,r_{2m} \le n} \nonumber\\
\qquad:= \frac{1}{n!} \sum_{\sigma_1 \in \frS_n} \cdots \sum_{\sigma_{2m} \in \frS_n} \left(\prod_{k=1}^{2m} \sgn(\sigma_k)\right) \cdot \prod_{j=1}^n A(\sigma_1(j),\dots,\sigma_{2m}(j)).\label{hyperdet}
\end{gather}
For simplicity, we will denote $\Det(A(r_1,\dots,r_{2m}))_{1 \le r_1,\dots,r_{2m} \le n}$ by $\Det(A)$ if there is no possibility of confusion.

The hyperdeterminant satisfies many properties that generalize the properties of the classical determinant. For a wide range of such properties, and numerous applications, we refer to~\cite{cayley46,cayley89,cayley45,evans,GKZ1,GKZ,hval,lecat19,oldenburger,rice,sokolov60, sokolov72}. For applications to mathematical physics and to the calculation of Selberg's famous integral, we refer to \cite{boussicault,luquethibon03,luquethibon}.

{\allowdisplaybreaks Let us express the formula \eqref{hyperdet} in two ways. First, observe that the sum over $\sigma_{2m}$ is a~classical determinant: For fixed permutations $\sigma_1,\dots,\sigma_{2m-1} \in \frS_n$,
\begin{gather*}
\sum_{\sigma_{2m} \in \frS_n} \sgn(\sigma_{2m}) \prod_{j=1}^n A(\sigma_1(j),\dots,\sigma_{2m}(j)) \\
\qquad\equiv \det(A(\sigma_1(i),\dots,\sigma_{2m-1}(i),j))_{1 \le i,j \le n} \\
\qquad= \sgn(\sigma_{2m-1}) \det\bigl(A\bigl(\sigma_1\sigma_{2m-1}^{-1}(i),\dots,\sigma_{2m-2}\sigma_{2m-1}^{-1}(i),i,j\bigr)\bigr)_{1 \le i,j \le 2}.
\end{gather*}
Therefore,
\begin{align*}
n! \Det(A) ={}& \sum_{\sigma_1,\dots,\sigma_{2m-1} \in \frS_n}
\left(\prod_{k=1}^{2m-2} \sgn(\sigma_k)\right)\\
&
\times \det\bigl(A(\sigma_1\sigma_{2m-1}^{-1}(i),\dots,\sigma_{2m-2}\sigma_{2m-1}^{-1}(i),i,j)\bigr)_{1 \le i,j \le n} .
\end{align*}
Replacing $\sigma_k$ by $\sigma_k \sigma_{2m-1}$, $1 \le k \le 2m-2$, we obtain
\begin{gather}
\Det(A)
= \sum_{\sigma_1,\dots,\sigma_{2m-2} \in \frS_n} \left(\prod_{k=1}^{2m-2} \sgn(\sigma_k)\right) \cdot \det(A(\sigma_1(i),\dots,\sigma_{2m-2}(i),i,j))_{1 \le i,j \le n}.\label{recurrencetwobytwo}
\end{gather}
This represents $\Det(A)$ as an alternating multisum of classical determinants.

}

The second way to express the hyperdeterminant is by rewriting \eqref{hyperdet} as
\begin{align*}
n! \Det(A)={}& \sum_{\sigma_{2m-1},\sigma_{2m} \in \frS_n} \prod_{k=2m-1}^{2m} \sgn(\sigma_k) \\
& \times \sum_{\sigma_1,\dots,\sigma_{2m-2} \in \frS_n} \prod_{k=1}^{2m-2} \sgn(\sigma_k) \cdot \prod_{j=1}^n A(\sigma_1(j),\dots,\sigma_{2m-2}(j),\sigma_{2m-1}(j),\sigma_{2m}(j)).
\end{align*}
For fixed $r_{2m-1},r_{2m} \in \{1,\dots,n\}$, define the multidimensional array
$
B_{r_{2m-1},r_{2m}}(r_1,\dots,r_{2m-2}) = A(r_1,\dots,r_{2m-2},r_{2m-1},r_{2m})
$,
where $r_1,\dots,r_{2m-2} \hspace{-0.15pt}\in\hspace{-0.15pt} \{1,\dots,n\}$. Then, for fixed ${\sigma_{2m-1},\sigma_{2m}\hspace{-0.15pt} \in\hspace{-0.15pt} \frS_n}$,
\begin{gather*}
 \sum_{\sigma_1,\dots,\sigma_{2m-2} \in \frS_n} \prod_{k=1}^{2m-2} \sgn(\sigma_k) \cdot \prod_{j=1}^n A(\sigma_1(j),\dots,\sigma_{2m-2}(j),\sigma_{2m-1}(j),\sigma_{2m}(j)) \\
\qquad= \sum_{\sigma_1,\dots,\sigma_{2m-2} \in \frS_n} \prod_{k=1}^{2m-2} \sgn(\sigma_k) \cdot \prod_{j=1}^n B_{\sigma_{2m-1}(j),\sigma_{2m}(j)}(\sigma_1(j),\dots,\sigma_{2m-2}(j)) \\
\qquad= \sum_{\sigma_1,\dots,\sigma_{2m-2} \in \frS_n} \prod_{k=1}^{2m-2} \sgn(\sigma_k) \cdot \prod_{j=1}^n B_{\sigma_{2m-1}\sigma_{2m}^{-1}(j),j}\bigl(\sigma_1\sigma_{2m}^{-1}(j),\dots,\sigma_{2m-2}\sigma_{2m}^{-1}(j)\bigr) \\
\qquad= \sum_{\sigma_1,\dots,\sigma_{2m-2} \in \frS_n} \prod_{k=1}^{2m-2} \sgn(\sigma_k) \cdot \prod_{j=1}^n B_{\sigma_{2m-1}\sigma_{2m}^{-1}(j),j}(\sigma_1(j),\dots,\sigma_{2m-2}(j)).
\end{gather*}
Therefore,
\begin{align*}
\Det(A) ={} & \frac{1}{n!} \sum_{\sigma_{2m-1},\sigma_{2m} \in \frS_n} \prod_{k=2m-1}^{2m} \sgn(\sigma_k) \\
& \times \sum_{\sigma_1,\dots,\sigma_{2m-2} \in \frS_n} \prod_{k=1}^{2m-2} \sgn(\sigma_k) \cdot \prod_{j=1}^n B_{\sigma_{2m-1}\sigma_{2m}^{-1}(j),j}(\sigma_1(j),\dots,\sigma_{2m-2}(j)).
\end{align*}
On replacing $\sigma_{2m-1}$ by $\sigma_{2m-1}\sigma_{2m}$, we obtain
\begin{align}
\Det(A) ={} & \sum_{\sigma_{2m-1} \in \frS_n} \sgn(\sigma_{2m-1}) \nonumber \\
& \times \sum_{\sigma_1,\dots,\sigma_{2m-2} \in \frS_n} \prod_{k=1}^{2m-2} \sgn(\sigma_k) \cdot \prod_{j=1}^n B_{\sigma_{2m-1}(j),j}(\sigma_1(j),\dots,\sigma_{2m-2}(j)) \nonumber \\
= {}& n! \sum_{\sigma_{2m-1} \in \frS_n} \sgn(\sigma_{2m-1})
\Det(B_{\sigma_{2m-1}(j),j}(\sigma_1(j),\dots,\sigma_{2m-2}(j)).\label{recurrence_hyperdet}
\end{align}
This represents $\Det(A)$ as a single alternating sum of hyperdeterminants.

\section{A Binet--Cauchy theorem for hyperdeterminants}
\label{sec_BC_for_hyperdets}

In reviewing the hyperdeterminantal literature, one observes that hyperdeterminants satisfy various generalizations of the classical Binet--Cauchy formula. Variants of the Binet--Cauchy formula for multidimensional arrays have appeared in the literature going back to \cite{gasparyan,hval}, and other analogs of the formula have appeared in \cite{barvinok,boussicault,dionisi,evans,karlinrinott88,matsumoto}. In all of those cited articles, the proof of the corresponding generalized Binet--Cauchy formula is analogous to the proof of the classical Binet--Cauchy formula \cite{Karlin} and also to generalizations of the formula that are based on the theory of finite reflection groups \cite{grossrichards89,grossrichards95}.

The statement and proof of the following generalized Binet--Cauchy formula are due to Matsumoto \cite[Proposition 2.1]{matsumoto}.

\begin{Proposition}[Matsumoto \cite{matsumoto}]
\label{prop_binet_cauchy}
Let $(\mcalM,\mu)$ be a sigma-finite measure space and let $\{\phi_{k,r}\mid\allowbreak 1 \le k \le 2m,\, 1 \le r \le n\}$ be a collection of complex-valued functions on $\mcalM$ such that, for all~${r_1,\dots,r_{2m}}$ satisfying $1 \le r_1,\dots,r_{2m} \le n$, the integral
\begin{equation}
A(r_1,\dots,r_{2m}) := \int_\mcalM \prod_{k=1}^{2m} \phi_{k,r_k}(x) \dd\mu(x),\label{BCkernel}
\end{equation}
converges absolutely. Then, with $\mcalM^n := \mcalM \times \cdots \times \mcalM$ $(n$ factors$)$,
\begin{equation}
\Det(A(r_1,\dots,r_{2m})) = \frac{1}{n!} \int_{\mcalM^n} \prod_{k=1}^{2m} \det(\phi_{k,r}(x_s))_{1 \le r,s \le n} \cdot \prod_{j=1}^n \dd\mu(x_j).\label{binetcauchy1}
\end{equation}
Moreover, if $\mcalM$ is totally ordered then
\begin{equation}
\Det(A(r_1,\dots,r_{2m})) = \idotsint\limits_{\substack{x_1 > \cdots > x_n {\phantom{\big|}}\\ x_1,\dots,x_n \in \mcalM}} \prod_{k=1}^{2m} \det(\phi_{k,r}(x_s))_{1 \le r,s \le n} \cdot \prod_{j=1}^n \dd\mu(x_j).\label{binetcauchy2}
\end{equation}
\end{Proposition}

\begin{proof}
By \eqref{hyperdet},
\begin{align*}
n! \Det(A(r_1,\dots,r_{2m})) &= \sum_{\sigma_1 \in \frS_n} \cdots \sum_{\sigma_{2m} \in \frS_n} \left(\prod_{k=1}^{2m} \sgn(\sigma_k)\right) \prod_{j=1}^n \left(\int_\mcalM \prod_{k=1}^{2m} \phi_{k,\sigma_k(j)}(x) \dd\mu(x)\right) \nonumber \\
&= \int_{\mcalM^n} \sum_{\sigma_1 \in \frS_n} \cdots \sum_{\sigma_{2m} \in \frS_n} \prod_{k=1}^{2m} \Biggl(\sgn(\sigma_k) \prod_{j=1}^n \phi_{k,\sigma_k(j)}(x_j)\Biggr) \prod_{j=1}^n \dd\mu(x_j).
\end{align*}
On writing the multi-sum over $\sigma_1,\dots,\sigma_{2m}$ as a product of summations, we obtain
\begin{align}
n! \Det(A(r_1,\dots,r_{2m})) &= \int_{\mcalM^n} \prod_{k=1}^{2m} \Biggl(\sum_{\sigma_k \in \frS_n} \sgn(\sigma_k) \prod_{j=1}^n \phi_{k,\sigma_k(j)}(x_j)\Biggr) \prod_{j=1}^n \dd\mu(x_j) \nonumber \\
&= \int_{\mcalM^n} \prod_{k=1}^{2m} \det(\phi_{k,r}(x_s))_{1\le r,s\le n} \prod_{j=1}^n \dd\mu(x_j),\label{eq_binetcauchy3}
\end{align}
and this establishes \eqref{binetcauchy1}.

Consider next the case in which $\mcalM$ is totally ordered. On observing that all $k$ determinants in the integrand in \eqref{eq_binetcauchy3} are identically zero if any two $x_s$ coincide, it follows that the integral in \eqref{eq_binetcauchy3} reduces to an integral over the union of disjoint open {\it Weyl chambers},
\[
\bigcup_{\sigma \in \frS_n} \{(x_1,\dots,x_n) \in \mcalM^n\mid x_{\sigma(1)} > \cdots > x_{\sigma(n)}\}.
\]
For each $\sigma \in \frS_n$, consider the corresponding Weyl chamber
\[
\CC_n(\sigma) := \{(x_1,\dots,x_n) \in \mcalM^n\mid x_{\sigma(1)} > \cdots > x_{\sigma(n)}\};
\]
then
\[
n! \Det(A(r_1,\dots,r_{2m})) = \sum_{\sigma \in \frS_n} \int_{\CC_n(\sigma)} \prod_{k=1}^{2m} \det(\phi_{k,r}(x_s))_{1\le r,s\le n} \prod_{j=1}^n \dd\mu(x_j).
\]
On the chamber $\CC_n(\sigma)$ we make the change-of-variables $(x_{\sigma(1)},\dots,x_{\sigma(n)}) \to (x_1,\dots,x_n)$, thereby obtaining
\begin{gather*}
\int_{\CC_n(\sigma)} \prod_{k=1}^{2m} \!\det(\phi_{k,r}(x_s))_{1\le r,s\le n} \prod_{j=1}^n \dd\mu(x_j) = \int_{\CC_n} \prod_{k=1}^{2m} (\sgn(\sigma) \det(\phi_{k,r}(x_s))_{1\le r,s\le n}) \prod_{j=1}^n\! \dd\mu(x_j),
\end{gather*}
where $\CC_n = \{(x_1,\dots,x_n) \in \mcalM^n\mid x_1 > \cdots > x_n\}$ is the {\it fundamental Weyl chamber}. Noting that there are $n!$ chambers, each of which corresponds to a unique permutation $\sigma \in \frS_n$, and also that $(\sgn(\sigma))^{2m} \equiv 1$, we obtain \eqref{binetcauchy2}.
\end{proof}

We now give an application of Proposition~\ref{prop_binet_cauchy} that will be crucial for establishing later the hyperdeterminantal total positivity properties of the kernel \eqref{eq_exp_htp}.

\begin{Example}
\label{exponentialhyperdet}
In Proposition~\ref{prop_binet_cauchy} set $\mcalM = \{0,1,2,\dots\}$, the set of nonnegative integers. Let~$\mu$ be the discrete (counting) measure with weight $\mu(\{i\}) = 1/i!$, $i \in \mcalM$. For $1 \le k \le 2m$, $1 \le r \le n$, and $i \in \mcalM$, let $\xx_k = (x_{k,1},\dots,x_{k,n}) \in \R^n$ and define $\phi_{k,r}(i) = x_{k,r}^i$. By \eqref{BCkernel},
\[
A(r_1,\dots,r_{2m}) = \sum_{i=0}^\infty \dfrac{x_{1,r_1}^i \cdots x_{2m,r_{2m}}^i}{i!}
= \exp(x_{1,r_1} \cdots x_{2m,r_{2m}}).
\]
On applying \eqref{binetcauchy2}, we obtain the hyperdeterminantal summation formula
\begin{equation}
\Det(\exp(x_{1,r_1} \cdots x_{2m,r_{2m}})) = \sum_{i_1 > \cdots > i_n \ge 0} \frac{1}{i_1! \cdots i_n!} \cdot \prod_{k=1}^{2m} \det\bigl(x_{k,r}^{i_s}\bigr)_{1 \le r,s\le n}.\label{expbinetcauchy}
\end{equation}

Define $\lambda_s = i_s - n + s$, $s=1,\dots,n$; then $\lambda_1 \ge \cdots \ge \lambda_n \ge 0$, so the vector $\lambda = (\lambda_1,\dots,\lambda_n)$ is a {\it partition}. Also let
\begin{equation}
V(\xx_k) = \prod_{1 \le r < s \le n} (x_{k,r} - x_{k,s}),\label{eq_vandermonde}
\end{equation}
be the Vandermonde determinant in the variables $\xx_k = (x_{k,1},\dots,x_{k,n})$; then
\begin{equation}
s_\lambda(\xx_k) = \frac{\det\bigl(x_{k,r}^{i_s}\bigr)_{1 \le r,s\le n}}{V(\xx_k)}\label{schur_function}
\end{equation}
is the well-known Schur function \cite[p.~40]{macdonald}. Writing \eqref{schur_function} in the form
\begin{equation}
\det\bigl(x_{k,r}^{i_s}\bigr)_{1 \le r,s\le n} = V(\xx_k) s_\lambda(\xx_k),\label{schurhyperdet}
\end{equation}
and applying this to \eqref{expbinetcauchy}, we obtain a summation formula that expresses the corresponding hyperdeterminant in terms of a sum of weighted products of Schur functions
\begin{equation}
\frac{\Det(\exp(x_{1,r_1} \cdots x_{2m,r_{2m}}))}{\prod_{k=1}^{2m} V(\xx_k)} = \sum_\lambda \prod_{j=1}^n \frac{1}{(\lambda_j + n-j)!} \cdot \prod_{k=1}^{2m} s_\lambda(\xx_k),\label{schurhyperdetsum}
\end{equation}
where the sum is taken over all partitions $\lambda = (\lambda_1,\dots,\lambda_n)$ of all nonnegative integers such that~$\lambda$ is of length $\ell(\lambda) \le n$.
\end{Example}

\section[A hyperdeterminantal extension of an integral of Harish-Chandra]{A hyperdeterminantal extension\\ of an integral of Harish-Chandra}
\label{sec_extension_HC}

In this section, we extend an integral of Harish-Chandra \cite{harishchandra} that has played a prominent role in the theory of total positivity \cite{grossrichards89,grossrichards95}. Further, we apply the extended integral to derive Schur function summation formulas for some hyperdeterminants defined in terms of the classical generalized hypergeometric series, and this will also connect the theory of hyperdeterminatal total positivity with the hypergeometric functions of Hermitian matrix argument.

Denote by $\mcalHn$ the space of $n \times n$ Hermitian matrices. In the sequel, if $X \in \mcalHn$ has eigenvalues $x_1,\dots,x_n$ and $\lambda$ is a partition then we will write $s_\lambda(X)$ for $s_\lambda(x_1,\dots,x_n)$, the Schur function with arguments $x_1,\dots,x_n$. Since every Hermitian matrix can be diagonalized by a~unitary transformation, then $s_\lambda(X)$ is a {\it unitarily invariant} polynomial function on $\mcalHn$.

For $X_1, X_2 \in \mcalHn$, we write $s_\lambda(X_1X_2)$ to denote $s_\lambda(w_1,\dots,w_n)$, where $w_1,\dots,w_n$ are the eigenvalues of $X_1X_2$. Since the matrices $X_1X_2$ and $X_2X_1$ have the same eigenvalues, then~${s_\lambda(X_1X_2) = s_\lambda(X_2X_1)}$.

For the case in which $m=1$, the summation formula \eqref{schurhyperdetsum} reduces to a sum that arose crucially in \cite{grossrichards89}, having been used there to derive a new proof that the fundamental classical example, $K(x,y) = \exp(xy)$, $(x,y) \in \R^2$, is strictly totally positive of order infinity (STP$_\infty$). Let~${
\beta_n = \prod_{j=1}^n (j-1)!}
$
and denote by $U(n)$ the group of $n \times n$ unitary matrices. For generic~${U \in U(n)}$ let $\dd U$ be the Haar measure on $U(n)$, normalized to have total volume $1$.

As noted in \cite{grossrichards89}, the special case with $m=1$ of \eqref{schurhyperdetsum} also appears prominently in an integral formula of Harish-Chandra \cite{harishchandra}:

\begin{Proposition}[Harish-Chandra \cite{harishchandra}]
\label{prop_hc_integral}
Suppose that $X_1, X_2 \in \mcalHn$ have eigenvalues $x_{1,1},\allowbreak x_{1,2}, \dots,x_{1,n}$ and $x_{2,1},x_{2,2},\dots,x_{2,n}$, respectively. Then
\begin{equation}
\frac{\det(\exp(x_{1,r_1} x_{2,r_2}))_{1 \le r_1,r_2\le n}}{V(x_1) V(x_2)} = \beta_n^{-1} \int_{U(n)} \exp\bigl(\tr UX_1U^{-1}X_2\bigr) \dd U.\label{hcintegraldet}
\end{equation}
\end{Proposition}

Setting $m=1$ in \eqref{schurhyperdetsum}, and then comparing the right-hand sides of \eqref{hcintegraldet} and \eqref{schurhyperdetsum}, it follows that
\begin{equation}
\int_{U(n)} \exp\bigl(\tr UX_1U^{-1}X_2\bigr) \dd U = \beta_n \sum_\lambda \prod_{j=1}^n \frac{1}{(\lambda_j + n-j)!} \cdot s_\lambda(x_1) s_\lambda(x_2).\label{hcintegralsum}
\end{equation}

As it turns out, an alternative proof of \eqref{hcintegralsum} was obtained in \cite{grossrichards89} by means of the theory of zonal polynomials and the hypergeometric functions of matrix argument. To extend the formula~\eqref{hcintegralsum} to the setting of hyperdeterminants, we will again apply the theory of zonal polynomials. In the proof of the next result, we will largely retain the notation of \cite[Section 2]{grossrichards89} and will follow closely the exposition given there. In particular, if $\lambda = (\lambda_1,\dots,\lambda_n)$ is a partition of length~${\ell(\lambda) \le n}$ then the {\it weight} of $\lambda$ is $|\lambda| = \lambda_1+\cdots+\lambda_n$.

For $a \in \C$ and $j = 0,1,2,\dots$, the {\it classical rising factorial} is defined as
\begin{equation}
(a)_j = \frac{\Gamma(a+j)}{\Gamma(a)} = a(a+1)\cdots (a+j-1).\label{classicalrising}
\end{equation}
For $a \in \C$ and any partition $\lambda$ of length $\ell(\lambda) \le n$, the {\it partitional rising factorial} is
\begin{equation}
(a)_\lambda = \prod_{j=1}^n (a-j+1)_{\lambda_j}.\label{partitionalrising}
\end{equation}
Define
\[
\omega_\lambda = |\lambda|!
\frac{\prod_{1 \le i<j \le n} (\lambda_i-\lambda_j-i+j)}{\prod_{j=1}^n (\lambda_j+n-j)!}
\]
and set
\[
d_\lambda := s_\lambda(I_n) = \frac{\prod_{1 \le i<j \le n} (\lambda_i-\lambda_j-i+j)}{\prod_{j=1}^n (j-1)!}.
\]
Then it is straightforward to see that
\begin{equation}
\omega_\lambda = \frac{|\lambda|! d_\lambda}{(n)_\lambda}.\label{omegadlambda}
\end{equation}
Further, for any $n \times n$ Hermitian matrix $X$ and partition $\lambda$, define the {\it zonal polynomial} as normalized Schur function,
\begin{equation}
Z_\lambda(X) = \omega_\lambda s_\lambda(X).\label{zonalschur}
\end{equation}
As shown in \cite[Section 2]{grossrichards89}, it is a consequence of \eqref{partitionalrising}--\eqref{omegadlambda} that, for all $l=0,1,2,\dots$, the normalization in \eqref{zonalschur} leads to the expansion
\begin{equation}
(\tr X)^l = \sum_{|\lambda| = l} Z_\lambda(X).\label{eq_tracepowers}
\end{equation}

The zonal polynomials $Z_\lambda$ satisfy the {\it mean-value property}: For any $n \times n$ Hermitian matrices~$X_1$ and $X_2$,
\begin{equation}
\int_{U(n)} Z_\lambda\bigl(UX_1U^{-1}X_2\bigr) \dd U = \frac{Z_\lambda(X_1) Z_\lambda(X_2)}{Z_\lambda(I_n)}.\label{meanvalue}
\end{equation}
The mean-value property also implies that
\begin{gather*}
\int_{U(n)} \int_{U(n)} Z_\lambda\bigl(U_1 X_1 U_1^{-1} U_2 X_2 U_2^{-1}\bigr) \dd U_1 \dd U_2 \\
\qquad\equiv \int_{U(n)} \int_{U(n)} Z_\lambda\bigl(U_1 X_1 U_1^{-1} \cdot U_2 X_2 U_2^{-1}\bigr) \dd U_1 \dd U_2 = \int_{U(n)} \frac{Z_\lambda(X_1) Z_\lambda\bigl(U_2 X_2 U_2^{-1}\bigr)}{Z_\lambda(I_n)} \dd U_2 \\
\qquad= \frac{Z_\lambda(X_1)}{Z_\lambda(I_n)} \int_{U(n)} Z_\lambda\bigl(U_2 X_2 U_2^{-1}\bigr) \dd U_2.
\end{gather*}
Since $s_\lambda$, and hence $Z_\lambda$, is unitarily invariant then $Z_\lambda\bigl(U_2 X_2 U_2^{-1}\bigr) = Z_\lambda(X_2)$ for all $U_2 \in U(n)$. Also since the Haar measure $\dd U_2$ is normalized then we obtain
\begin{equation}
\int_{U(n)} \int_{U(n)} Z_\lambda\bigl(U_1 X_1 U_1^{-1} U_2 X_2 U_2^{-1}\bigr) \dd U_1 \dd U_2 = \frac{Z_\lambda(X_1) Z_\lambda(X_2)}{Z_\lambda(I_n)}.\label{eq_iterated_mvp}
\end{equation}
We now state and prove the promised hyperdeterminantal generalization of \eqref{hcintegralsum}.

\begin{Theorem}
\label{thm_extd_hc}
For $k=1,\dots,2m$, let $X_k$ be an $n \times n$ Hermitian matrix with eigenvalues $x_{k,1},\dots,x_{k,n}$, and denote $s_\lambda(x_{k,1},\dots,x_{k,n})$ by $s_\lambda(X_k)$. Then
\begin{gather}
\int_{U(n)} \cdots \int_{U(n)} \exp\left(\tr \prod_{k=1}^{2m} U_k X_k U_k^{-1}\right) \prod_{k=1}^{2m} \dd U_k \nonumber\\
\qquad= \beta_n \sum_\lambda \frac{1}{\prod_{j=1}^n (\lambda_j + n-j)!} \cdot \frac{1}{(s_\lambda(I_n))^{2m-2}} \prod_{k=1}^{2m} s_\lambda(X_k).\label{hcextended}
\end{gather}
\end{Theorem}

\begin{proof}
It follows from \eqref{eq_tracepowers} that for any matrix $X$ that is a product of Hermitian matrices,
\[
\exp(\tr X) = \sum_{k=0}^\infty \frac{(\tr X)^k}{k!} = \sum_\lambda \frac{1}{|\lambda|!} Z_\lambda(X).
\]
Applying this formula to expand the integrand in \eqref{hcextended}, we obtain
\begin{equation}
\exp\left(\tr \prod_{k=1}^{2m} U_k X_k U_k^{-1}\right) = \sum_\lambda \frac{1}{|\lambda|!} Z_\lambda\left(\prod_{k=1}^{2m} U_k X_k U_k^{-1}\right).\label{eq_exp_tr_expan}
\end{equation}
We now apply the mean-value property \eqref{meanvalue} to integrate iteratively with respect to the normalized Haar measures $\dd U_1,\dots,\dd U_{2m}$, as in the derivation of \eqref{eq_iterated_mvp}. Then we obtain
\begin{equation}
\int_{U(n)} \cdots \int_{U(n)} Z_\lambda\left(\prod_{k=1}^{2m} U_kX_kU_k^{-1}\right) \prod_{k=1}^{2m} \dd U_k = Z_\lambda(X_1) \prod_{k=2}^{2m} \frac{Z_\lambda(X_k)}{Z_\lambda(I_n)},\label{extendedmeanvalue}
\end{equation}
and it follows from \eqref{eq_exp_tr_expan} and \eqref{extendedmeanvalue} that
\begin{equation}
\int_{U(n)} \cdots \int_{U(n)} \exp\left(\tr \prod_{k=1}^{2m} U_kX_kU_k^{-1}\right) \prod_{k=1}^{2m} \dd U_k = \sum_\lambda \frac{1}{|\lambda|!} Z_\lambda(X_1) \prod_{k=2}^{2m} \frac{Z_\lambda(X_k)}{Z_\lambda(I_n)}.\label{hcprelim}
\end{equation}
Now using \eqref{zonalschur} to express each $Z_\lambda$ in terms of $s_\lambda$, and substituting from \eqref{omegadlambda} for $\omega_\lambda$, the right-hand side of \eqref{hcprelim} becomes
\begin{align}
\sum_\lambda \frac{\omega_\lambda}{|\lambda|!} s_\lambda(X_1) \prod_{k=2}^{2m} \frac{s_\lambda(X_k)}{s_\lambda(I_n)}
&= \sum_\lambda \frac{s_\lambda(I_n)}{(n)_\lambda} s_\lambda(X_1) \prod_{k=2}^{2m} \frac{s_\lambda(X_k)}{s_\lambda(I_n)} \nonumber \\
&= \sum_\lambda \frac{1}{(n)_\lambda} \frac{s_\lambda(X_1) \cdots s_\lambda(X_{2m})}{(s_\lambda(I_n))^{2m-2}}.\label{hcprelimsum}
\end{align}
Next, it is straightforward to verify that
\begin{align*}
(n)_\lambda &= \prod_{j=1}^n (n-j+1)_{\lambda_j}
= \prod_{j=1}^n \frac{\Gamma(n+\lambda_j-j+1)}{\Gamma(n-j+1)}
= \beta_n^{-1} \prod_{j=1}^n (\lambda_j+n-j)!,
\end{align*}
and by substituting this expression for $(n)_\lambda$ into \eqref{hcprelimsum}, we obtain \eqref{hcextended}.
\end{proof}

As a further generalization of Example \ref{exponentialhyperdet}, we have the following.

\begin{Example}
\label{hgfhyperdet}
Let $X = \{0,1,2,\dots\}$, the set of nonnegative integers. Let $\xx_k = (x_{k,1},\dots,x_{k,n}) \in \C^n$, $k=1,\dots,2m$ and, for $i \in X$, define $\phi_{k,r}(i) = x_{k,r}^i$, $1 \le k \le 2m$, $1 \le r \le n$. For nonnegative integers $p$ and $q$, let $a_1,\dots,a_p,b_1,\dots,b_q \in \C$ such that $-b_t+j-1$ is not a nonnegative integer for all $t=1,\dots,q$ and $j=1,\dots,n$.

Define the discrete measure,
\smash{$
\mu(\{i\}) = \frac{(a_1)_i \cdots (a_p)_i}{(b_1)_i \cdots (b_q)_i} \frac{1}{i!}$},
$i \in X$, and also define the multidimensional array $A(r_1,\dots,r_{2m})$ according to \eqref{BCkernel} so that
\begin{align}
A(r_1,\dots,r_{2m}) & = \sum_{i=0}^\infty \frac{(a_1)_i \cdots (a_p)_i}{(b_1)_i \cdots (b_q)_i} \frac{(x_{1,r_1} \cdots x_{2m,r_{2m}})^i}{i!} \nonumber\\
& \equiv {}_pF_q(a_1,\dots,a_p;b_1,\dots,b_q;x_{1,r_1} \cdots x_{2m,r_{2m}}),\label{pfqarray}
\end{align}
where
\begin{equation}
\label{classicalargpfq}
{}_pF_q(a_1,\dots,a_p;b_1,\dots,b_q;x) = \sum_{i=0}^\infty \frac{(a_1)_i \cdots (a_p)_i}{(b_1)_i \cdots (b_q)_i} \frac{x^i}{i!},
\end{equation}
$x \in \C$, is the standard notation for the classical generalized hypergeometric series \cite{andrews}.

The convergence properties of the series \eqref{pfqarray} follow from standard criteria for convergence of the classical generalized hypergeometric series~\cite[Theorem 2.1.1]{andrews}: For $p \le q$, the series~\eqref{pfqarray} converges for all $x_{1,r_1},\dots,x_{2m,r_{2m}} \in \C$; if $p = q+1$, then the series converges whenever $|x_{1,r_1} \cdots x_{2m,r_{2m}}| < 1$; and if $p > q+1$, then the series diverges for all $x_{1,r_1} \cdots x_{2m,r_{2m}} \neq 0$ unless it is a terminating series.

Next, we apply \eqref{schurhyperdet}, {\it viz.},
$
\det\bigl(x_{k,r}^{i_s}\bigr)_{1\le r,s\le n} = V(\xx_k) s_\lambda(\xx_k)$,
to obtain
\begin{gather}
\frac{\Det({}_pF_q(a_1,\dots,a_p;b_1,\dots,b_q;x_{1,r_1} \cdots x_{2m,r_{2m}}))}{\prod_{k=1}^{2m} V(\xx_k)}\nonumber \\
\qquad= \sum_\lambda \prod_{j=1}^n \frac{(a_1)_{\lambda_j+n-j} \cdots (a_p)_{\lambda_j+n-j}}{(b_1)_{\lambda_j+n-j} \cdots (b_q)_{\lambda_j+n-j}} \cdot \frac{1}{(\lambda_j+n-j)!} \cdot \prod_{k=1}^{2m} s_\lambda(\xx_k).\label{eq_hgfhyperdet_exp}
\end{gather}
It was assumed earlier that $-b_t+j-1$ is not a nonnegative integer for all $t=1,\dots,q$ and all~${j=1,\dots,n}$; that assumption is necessary to ensure that $(b_t)_{\lambda_j+n-j} \neq 0$ for all $t=1,\dots,q$ and $j=1,\dots,n$, in which case the expansion in \eqref{eq_hgfhyperdet_exp} is well defined.

By \eqref{classicalrising} and \eqref{partitionalrising},
\begin{align*}
\prod_{j=1}^n (a)_{\lambda_j+n-j} &= \prod_{j=1}^n \frac{\Gamma(a+\lambda_j+n-j)}{\Gamma(a)} \\
&= \prod_{j=1}^n (a+n-j)_{\lambda_j} \frac{\Gamma(a+n-j)}{\Gamma(a)}
= (a+n-1)_\lambda \prod_{j=1}^n (a)_{n-j},
\end{align*}
so we obtain
\begin{gather}
 \frac{\Det({}_pF_q(a_1,\dots,a_p;b_1,\dots,b_q;x_{1,r_1} \cdots x_{2m,r_{2m}}))}{\prod_{k=1}^{2m} V(\xx_k)} \nonumber \\
\qquad
= \prod_{j=1}^n \frac{(a_1)_{n-j} \cdots (a_p)_{n-j}}{(b_1)_{n-j} \cdots (b_q)_{n-j}} \nonumber \\
\phantom{\qquad
=}{} \times \sum_\lambda \frac{(a_1+n-1)_\lambda \cdots (a_p+n-1)_\lambda}{(b_1+n-1)_\lambda \cdots (b_q+n-1)_\lambda} \cdot \prod_{j=1}^n \frac{1}{(\lambda_j+n-j)!} \cdot \prod_{k=1}^{2m} s_\lambda(\xx_k).\label{pfqexample}
\end{gather}

For $(p,q) = (0,0)$, since ${}_0F_0(x) = \exp(x)$, $x \in \C$ (see \eqref{pfqarray}), then \eqref{pfqexample} reduces to \eqref{schurhyperdetsum} and we recover Example \ref{exponentialhyperdet}.

For $(p,q) = (1,0)$, we apply the classical negative-binomial theorem, ${}_1F_0(a;x) = (1-x)^{-a}$, $|x| < 1$, and then we obtain from \eqref{pfqexample} the formula
\[
\frac{\Det((1-x_{1,r_1} \cdots x_{2m,r_{2m}})^{-a})}{\prod_{k=1}^{2m} V(\xx_k)}
= \prod_{j=1}^n (a)_{n-j} \cdot \sum_\lambda (a+n-1)_\lambda \prod_{j=1}^n \frac{1}{(\lambda_j+n-j)!} \cdot \prod_{k=1}^{2m} s_\lambda(\xx_k).
\]
For the case in which $a = 1$, this identity reduces to \cite[Corollary 3.2]{matsumoto}.
\end{Example}

We now extend Theorem \ref{thm_extd_hc} to arbitrary generalized hypergeometric functions of Hermitian matrix argument \cite{grossrichards87,grossrichards89}. Let $a_1,\dots,a_p$ and $b_1,\dots,b_q$ be complex numbers such that $-b_t+j-1$ is not a nonnegative integer for all $t=1,\dots,q$ and $j=1,\dots,n$. The generalized hypergeometric function of an $n \times n$ Hermitian matrix argument $X$ is defined by the zonal polynomial series
\begin{equation}
\label{matrixargpfq}
{}_pF_q(a_1,\dots,a_p;b_1,\dots,b_q;X) = \sum_{l=0}^\infty \frac{1}{l!} \sum_{|\lambda|=l} \frac{(a_1)_\lambda\cdots (a_p)_\lambda}{(b_1)_\lambda \cdots (b_q)_\lambda} Z_\lambda(X),
\end{equation}
where the inner sum is over all partitions $\lambda$ of all nonnegative integers such that $\lambda$ is of length~${\ell(\lambda) \le n}$ and weight $|\lambda|=l$.

For the case in which $n = 1$, the series \eqref{matrixargpfq} reduces to the classical, scalar-argument case in \eqref{classicalargpfq}. Although the notation ${}_pF_q$ is used in both the scalar and matrix argument cases, its usage will always be clear from the context.

The convergence properties of the series \eqref{matrixargpfq}, established in \cite[Section 6]{grossrichards87} in terms of $\|X\|$, the spectral norm of $X$, are as follows: For $p \le q$, the series \eqref{matrixargpfq} converges for all $\|X\| < \infty$; if~${p = q+1}$, then the series converges whenever $\|X\| < 1$; and if $p > q+1$, then the series diverges unless it is terminating.

If $(p,q) = (0,0)$, then by \eqref{eq_tracepowers}
\begin{equation}
\label{eq_0F0}
{}_0F_0(X) = \sum_{l=0}^\infty \frac{1}{l!} \sum_{|\lambda|=l} Z_\lambda(X) = \sum_{l=0}^\infty \frac{1}{l!} (\tr X)^l = \exp(\tr X).
\end{equation}
For $(p,q) = (1,0)$, $a \in \C$, and $\|X\| < 1$, it is well known (see, e.g., \cite{grossrichards87,grossrichards89}) that
\begin{equation}
\label{eq_1F0}
{}_1F_0(a;X) = \det(I_n-X)^{-a},
\end{equation}
where $I_n$ denotes the $n \times n$ identity matrix.

The following result extends Theorem \ref{thm_extd_hc} to any generalized hypergeometric functions of matrix argument. The proof of this result is obtained by expanding the matrix argument ${}_pF_q$ function and applying the same integration and simplification procedures used in \eqref{eq_exp_tr_expan}--\eqref{hcprelim} in the course of proving Theorem \ref{thm_extd_hc}, together with the iterative integration argument used to derive \eqref{eq_iterated_mvp}.

\begin{Theorem}
\label{ghf-extendedpfq}
For $k=1,\dots,2m$, let $X_k$ be an $n \times n$ Hermitian matrix with eigenvalues $x_{k,1},\dots,x_{k,n}$. Then
\begin{gather}
\int_{U(n)} \cdots \int_{U(n)} {}_pF_q\left(a_1,\dots,a_p;b_1,\dots,b_q;\prod_{k=1}^{2m} U_k X_k U_k^{-1}\right) \prod_{k=1}^{2m} \dd U_k \nonumber\\
\qquad= \beta_n \sum_\lambda \frac{(a_1)_\lambda \cdots (a_p)_\lambda}{(b_1)_\lambda \cdots (b_q)_\lambda} \frac{1}{\prod_{j=1}^n (\lambda_j + n-j)!} \frac{1}{(s_\lambda(I_n))^{2m-2}} \prod_{k=1}^{2m} s_\lambda(X_k).\label{hcextendedpfq}
\end{gather}
\end{Theorem}

On setting $(p,q)=(0,0)$ in \eqref{hcextendedpfq} and applying \eqref{eq_0F0}, we recover Theorem \ref{thm_extd_hc}. Further, if~${(p,q)=(1,0)}$, then by applying \eqref{eq_1F0} we obtain
\begin{gather*}
\int_{U(n)} \cdots \int_{U(n)} \det\left(I_n - \prod_{k=1}^{2m} U_k X_k U_k^{-1}\right)^{-a} \prod_{k=1}^{2m} \dd U_k \\
\qquad= \beta_n \sum_\lambda (a)_\lambda \frac{1}{\prod_{j=1}^n (\lambda_j + n-j)!} \frac{1}{(s_\lambda(I_n))^{2m-2}} \prod_{k=1}^{2m} s_\lambda(X_k),
\end{gather*}
with convergence for $\max\{\|X_1\|,\dots\|X_{2m}\|\} < 1$.

\section{Hyperdeterminantal total positivity}\label{sec_HTP}

In this section, we define the concept hyperdeterminantal total positivity, construct several examples of such kernels, and generalize the classical basic composition formula. Further, we apply the hyperdeterminantal Binet--Cauchy formula and the generalized basic composition formula to show how examples of HTP kernels can be constructed.

\subsection{Definition of hyperdeterminantal total positivity}

We now define the concept of hyperdeterminantal total positivity of general order $d$ for a kernel~${K\colon \R^{2m} \to [0,\infty)}$. As in Section \ref{sec_BC_for_hyperdets}, we use the notation
\[
\CC_n = \{(x_1,\dots,x_n) \in \R^n\mid x_1 > \cdots > x_n\}
\]
for the fundamental Weyl chamber in $\R^n$.

\begin{Definition}
\label{htpdefinition}
A kernel $K\colon \R^{2m} \to [0,\infty)$ is {\it hyperdeterminantal totally positive of order} $d$ (HTP$_d$) if, for any collection of vectors $\xx_k = (x_{k,1},\dots,x_{k,n}) \in \CC_n$, $1 \le k \le 2m$, we have
\begin{equation}
\label{htphyperdet}
\Det(K(x_{1,r_1},x_{2,r_2},\dots,x_{2m,r_{2m}}))_{1 \le r_1,\dots,r_{2m} \le n} \ge 0
\end{equation}
for all $n = 1,\dots,d$. If the hyperdeterminants \eqref{htphyperdet} are nonnegative for all $n \in \N$, then we say that $K$ is {\it hyperdeterminantal totally positive of order} $\infty$ (HTP$_\infty$).

If the hyperdeterminants \eqref{htphyperdet} are positive for all $n = 1,\dots,d$, then we say that $K$ is {\it hyperdeterminantal strictly totally positive of order} $d$ (HSTP$_d$). Similarly, if strict inequality holds for all $n \in \N$, then we say that $K$ is {\it hyperdeterminantal strictly totally positive of order infinity} (HSTP$_\infty$).
\end{Definition}

We remark that Definition \ref{htpdefinition} was given in \cite{Johnson_Richards} and, independently, in \cite{wanglixu}. Intuitively, a~kernel $K\colon\R^{2m} \to [0,\infty)$ is HTP$_d$ if the hyperdeterminant of every sub-array of the array
\begin{equation}
\label{eq_array}
(K(x_{1,r_1},x_{2,r_2},\dots,x_{2m,r_{2m}}))_{1 \le r_1,\dots,r_{2m} \le d}
\end{equation}
is nonnegative for all vectors $\xx_k = (x_{k,1},\dots,x_{k,n}) \in \CC_d$, $k=1,\dots,2m$. Note that the concept of a {\it minor} of an array was treated in \cite{boussicault} and, when expressed in that terminology, a~kernel~${K\colon\R^{2m} \to [0,\infty)}$ is HTP$_d$ if all minors of the array \eqref{eq_array} are nonnegative for all vectors~${\xx_1,\dots,\xx_{2m} \in \CC_d}$.

\subsection{Examples of HTP kernels}

\begin{Example}
\label{classicaltp}
The case $m = 1$: By the definition in \eqref{hyperdet} of the hyperdeterminant, a kernel~${K\colon \R^2 \to [0,\infty)}$ is HTP$_d$ if, for all $n = 1,\dots,d$, the determinant
\[
\det(K(x_{1,r_1},x_{2,r_2}))_{1 \le r_1,r_2 \le n} =
\begin{vmatrix}
K(x_{1,1},x_{2,1}) & K(x_{1,1},x_{2,2}) & \cdots & K(x_{1,1},x_{2,n}) \\
K(x_{1,2},x_{2,1}) & K(x_{1,2},x_{2,2}) & \cdots & K(x_{1,2},x_{2,n}) \\
\vdots & \vdots & & \vdots \\
K(x_{1,n},x_{2,1}) & K(x_{1,n},x_{2,2}) & \cdots & K(x_{1,n},x_{2,n})
\end{vmatrix}
\]
is nonnegative whenever $x_{1,1} > x_{1,2} > \cdots > x_{1,d}$ and $x_{2,1} > x_{2,2} > \cdots > x_{2,d}$. This condition is well known to be the classical definition of total positivity of order $d$.
\end{Example}

For $m \ge 2$, the concept of HTP$_d$ appears to be new even for $d=2$, and the following example illustrates the difference between HTP and the concept of multivariate total positivity developed in \cite{karlinrinott80}.

\begin{Example}
\label{classicalhtp}
The case $d = m = 2$: For vectors $\xx_k = (x_{k,1},x_{k,2}) \in \CC_2$, $k = 1,\dots,4$, we apply~\eqref{recurrencetwobytwo} to obtain
\begin{gather}
\Det(K(x_{1,r_1},\dots, x_{4,r_4}))_{1 \le r_1,\dots,r_4 \le 2} \nonumber \\
\qquad=
\begin{vmatrix}
K(x_{1,1},x_{2,1},x_{3,1},x_{4,1}) & K(x_{1,1},x_{2,1},x_{3,1},x_{4,2}) \\
K(x_{1,2},x_{2,2},x_{3,2},x_{4,1}) & K(x_{1,2},x_{2,2},x_{3,2},x_{4,2})
\end{vmatrix} \nonumber \\
\phantom{\qquad= }{} -
\begin{vmatrix}
K(x_{1,1},x_{2,2},x_{3,1},x_{4,1}) & K(x_{1,1},x_{2,2},x_{3,1},x_{4,2}) \\
K(x_{1,2},x_{2,1},x_{3,2},x_{4,1}) & K(x_{1,2},x_{2,1},x_{3,2},x_{4,2})
\end{vmatrix} \nonumber \\
\phantom{\qquad= }{} -
\begin{vmatrix}
K(x_{1,2},x_{2,1},x_{3,1},x_{4,1}) & K(x_{1,2},x_{2,1},x_{3,1},x_{4,2}) \\
K(x_{1,1},x_{2,2},x_{3,2},x_{4,1}) & K(x_{1,1},x_{2,2},x_{3,2},x_{4,2})
\end{vmatrix} \nonumber \\
\phantom{\qquad= }{} +
\begin{vmatrix}
K(x_{1,2},x_{2,2},x_{3,1},x_{4,1}) & K(x_{1,2},x_{2,2},x_{3,1},x_{4,2}) \\
K(x_{1,1},x_{2,1},x_{3,2},x_{4,1}) & K(x_{1,1},x_{2,1},x_{3,2},x_{4,2})
\end{vmatrix}.\label{eq_hyperdet_22a}
\end{gather}
Define the vectors
\[
\uu_1 = (x_{1,1},x_{2,2}), \qquad \uu_2 = (x_{1,2},x_{2,1}), \qquad
\vv_1 = (x_{3,1},x_{4,2}), \qquad \vv_2 = (x_{3,2},x_{4,1}).
\]
Since each $\xx_k \in \CC_2$, $k=1,\dots,4$, then
\begin{gather*}
\uu_1 \vee \uu_2 = (x_{1,1},x_{2,1}), \qquad \uu_1 \wedge \uu_2 = (x_{1,2},x_{2,2}), \\
\vv_1 \vee \vv_2 = (x_{3,1},x_{4,1}), \qquad \vv_1 \wedge \vv_2 = (x_{3,2},x_{4,2}).
\end{gather*}
Therefore, \eqref{eq_hyperdet_22a} becomes
\begin{gather}
\Det(K(x_{1,r_1},\dots, x_{4,r_4}))_{1 \le r_1,\dots,r_4 \le 2} \nonumber \\
\qquad=
\begin{vmatrix}
K(\uu_1 \vee \uu_2,\vv_1 \vee \vv_2) & K(\uu_1 \vee \uu_2,\vv_1) \\
K(\uu_1 \wedge \uu_2,\vv_2) & K(\uu_1 \wedge \uu_2,\vv_1 \wedge \vv_2)
\end{vmatrix} \nonumber \\
\phantom{\qquad= }{} - \begin{vmatrix}
K(\uu_1,\vv_1 \vee \vv_2) & K(\uu_1,\vv_1) \\
K(\uu_2,\vv_2) & K(\uu_2,\vv_1 \wedge \vv_2)
\end{vmatrix} \nonumber \\
\phantom{\qquad= }{} - \begin{vmatrix}
K(\uu_2,\vv_1 \vee \vv_2) & K(\uu_2,\vv_1) \\
K(\uu_1,\vv_2) & K(\uu_1,\vv_1 \wedge \vv_2)
\end{vmatrix} \nonumber \\
\phantom{\qquad= }{} + \begin{vmatrix}
K(\uu_1 \wedge \uu_2,\vv_1 \vee \vv_2) & K(\uu_1 \wedge \uu_2,\vv_1) \\
K(\uu_1 \vee \uu_2,\vv_2) & K(\uu_1 \vee \uu_2,\vv_1 \wedge \vv_2)
\end{vmatrix} .\label{eq_hyperdet_22b}
\end{gather}
Let $\ww_1 = (\uu_1 \vee \uu_2,\vv_1)$, $\ww_2 = (\uu_1 \wedge \uu_2,\vv_2)$. Then, $\ww_1 \vee \ww_2 = (\uu_1 \vee \uu_2,\vv_1 \vee \vv_2)$, $\ww_1 \wedge \ww_2 = (\uu_1 \wedge \uu_2,\vv_1 \wedge \vv_2)$, and the first determinant on the right-hand side of \eqref{eq_hyperdet_22b} equals
\begin{equation}
\label{eq_MTP_det1}
\begin{vmatrix}
K(\ww_1 \vee \ww_2) & K(\ww_1) \\
K(\ww_2) & K(\ww_1 \wedge \ww_2)
\end{vmatrix}.
\end{equation}
Therefore, if $K$ is multivariate TP$_2$ in the sense defined in~\cite{karlinrinott80} then it follows that the determinant \eqref{eq_MTP_det1} is nonnegative. However, for arbitrary multivariate TP$_2$ kernels, it appears to be difficult to derive general criteria that imply the nonnegativity of the alternating sum of all four determinants in~\eqref{eq_hyperdet_22b}.
\end{Example}

Turning to the construction of HTP$_d$ kernels, we remarked in the introduction that the fundamental example of a classical STP$_\infty$ kernel is the function $K(x_1,x_2) = \exp(x_1 x_2)$, $(x_1,x_2) \in \R^2$. We now generalize this example to the setting of hyperdeterminantal total positivity.

Define a {\it slice of the array $A = (A(r_1,\dots,r_{2m}))_{1 \le r_1,\dots,r_{2m} \le n}$ in the $j$th direction} to be the subset of all indices $(r_1,\dots,r_{2m})$ where the $j$th index, $r_j$, is fixed. Gelfand, \textit{et al.} \cite[Corollary~1.5\,(b)]{GKZ1} used the action of certain general linear groups to show that $\Det(A)$ is a homogeneous polynomial in the entries of each slice; this homogeneity property can also be deduced directly, albeit with more work, from the expansions \eqref{recurrencetwobytwo} and \eqref{recurrence_hyperdet}.

\begin{Proposition}
\label{prop_exponentialexample}
The following kernels are HSTP$_\infty$:
\begin{itemize}\itemsep=0pt
\item[$(i)$] $K_1(x_1,\dots,x_{2m}) = \exp(x_1 \cdots x_{2m})$, $x_1,\dots,x_{2m} \in \R$.

\item[$(ii)$] $K_2(x_1,x_2,\dots,x_{2m}) = x_1^{x_2 x_3 \cdots x_{2m}}$, $x_1,x_2,\dots,x_{2m} > 0$.
\end{itemize}
\end{Proposition}

\begin{proof}
(i) Let $\xx_k = (x_{k,1},\dots,x_{k,n})$, $k=1,\dots,2m$, and suppose that each $\xx_k \in [0,\infty)^n$, the nonnegative orthant in $\R^n$. By
\eqref{schurhyperdetsum},
\begin{equation}
\Det(\exp(x_{1,r_1} x_{2,r_2} \cdots x_{2m,r_{2m}})) = \left(\prod_{k=1}^{2m} V(\xx_k)\right) \cdot \sum_\lambda \prod_{j=1}^n \frac{1}{(\lambda_j + n-j)!} \cdot \prod_{k=1}^{2m} s_\lambda(\xx_k),\label{expsummation}
\end{equation}
where the sum is over all partitions $\lambda = (\lambda_1,\dots,\lambda_n)$ of length $\ell(\lambda) \le n$.
It is a well-known result~\cite{grossrichards89,macdonald} that $s_\lambda(\xx_k) \ge 0$ for all $\xx_k \in [0,\infty)^n$, and therefore each summand on the right-hand side of \eqref{expsummation} is nonnegative. Moreover, the sum is positive since $s_{(0)}(\xx_k) \equiv 1$, where $(0)$ denotes the zero partition.

Suppose also that each $\xx_k \in \CC_n$. Then, by \eqref{eq_vandermonde}, the Vandermonde polynomials $V(\xx_k)$ on the right-hand side of \eqref{expsummation} are positive. Therefore, the right-hand side of \eqref{expsummation} is positive for~${\xx_k \in \CC_n \cap [0,\infty)^n}$, $k=1,\dots,2m$, so we obtain
\[
\Det(\exp(x_{1,r_1} x_{2,r_2} \cdots x_{2m,r_{2m}})) > 0.
\]
This proves that the kernel $K_1$ is SHTP$_n$, and since $n$ was chosen arbitrarily then it follows that~$K_1$ is HSTP$_\infty$.

Now consider the case in which $x_{k,n} < 0$ for some $k$. For each such $k$, we multiply the corresponding slice of the array, {\it viz.},
\[
\{(\exp(x_{1,r_1} x_{2,r_2} \cdots x_{2m,r_{2m}})) \mid 1 \le i_1,\dots,i_{2m} \le n,\, i_k = n\}
\]
by exponential factors as follows: Form the vectors $x_{k,n}x_i = (x_{k,n}x_{k,1},\dots,x_{k,n}x_{in})$, $1 \le i \le n$; next, multiply the corresponding slice by
\[
\exp(-x_{1,i_1}x_{2,i_2} \cdots x_{k-1,i_{k-1}}x_{k,n}x_{k+1,i_{k+1}}\cdots x_{2m,i_{2m}}).
\]
thereby changing the generic entries of that slice to
\[
\exp(-x_{1,i_1}x_{2,i_2} \cdots x_{k-1,i_{k-1}}(x_{k,i_k}-x_{k,n})x_{k+1,i_{k+1}}\cdots x_{2m,i_{2m}}).
\]
On carrying out these multiplications for each $k$ such that $x_{k,n} < 0$, the outcome is that we have created a new multidimensional array, $(\exp(y_{1,r_1} y_{2,r_2} \cdots y_{2m,r_{2m}}))$, where
\[
y_{k,r_k} =
\begin{cases}
x_{k,r_k} & \hbox{if } x_{k,n} \ge 0, \\
x_{k,r_k} - x_{k,n} & \hbox{if } x_{k,n} < 0.
\end{cases}
\]
Noting that $y_{k,r_k} \ge 0$ for all $k=1,\dots,2m$, it then follows from the preceding argument that~${\Det(\exp(y_{1,r_1} y_{2,r_2} \cdots y_{2m,r_{2m}})) > 0}$.

As noted earlier, the hyperdeterminant is a homogeneous polynomial in the entries of each slice. Hence $\Det(\exp(y_{1,r_1} y_{2,r_2} \cdots y_{2m,r_{2m}}))$ and $\Det(\exp(x_{1,r_1} x_{2,r_2} \cdots x_{2m,r_{2m}}))$ differ only by powers of the multiplying factors. Since those factors are all exponential terms then the hyperdeterminants $\Det(\exp(y_{1,r_1} y_{2,r_2} \cdots y_{2m,r_{2m}}))$ and $\Det\bigl((x_{1,r_1} x_{2,r_2} \cdots x_{2m,r_{2m}}))$ differ only by positive factors, and therefore $\Det(\exp(x_{1,r_1} x_{2,r_2} \cdots x_{2m,r_{2m}})) > 0$.

(ii) Suppose that $x_1,x_2,\dots,x_{2m} > 0$. We now make the change-of-variables $x_1 \to \log x_1$ in the kernel in (i). Then the resulting kernel is $K_2(x_1,x_2,\dots,x_{2m}) = x_1^{x_2 x_3 \cdots x_{2m}}$, and it follows from the HDTP$_\infty$ property of $K_1$ and strictly increasing nature of the mapping $x_1 \to \log x_1$ that $K_2$ also is HDTP$_\infty$.
\end{proof}

\begin{Remark}
For the case in which some of the $x_j$ are integer variables, a kernel similar to $K_1$ was studied in \cite{wanglixu,xu}.

Also, in Proposition~\ref{prop_exponentialexample}, if $x_1,x_2,\dots,x_{2m} > 0$ and we make the changes-of-variables $x_j \to \exp(x_j)$, $j=1,\dots,2m$, then the resulting kernel is $K_3(x_1,\dots,x_{2m}) = \exp\bigl(e^{x_1 + \cdots + x_{2m}}\bigr)$, where $x_1,x_2,\dots,x_{2m} \in \R$. As $K_3$ depends on $x_1 + \cdots + x_{2m}$ only, then it generalizes the P\'olya frequency function type of TP kernels. See \cite{luquethibon,wanglixu,xu} for kernels of this type.
\end{Remark}

\begin{Example}
Consider the classical case in which $m = 1$. We want to show that if $x_{1,1} > \cdots > x_{1,n}$ and $x_{2,1} > \cdots > x_{2,n}$, then the determinant
\[
\begin{vmatrix}
\exp(x_{1,1}x_{2,1}) & \exp(x_{1,1}x_{2,2}) & \cdots & \exp(x_{1,1}x_{2,n}) \\
\vdots & \vdots & & \vdots \\
\exp(x_{1,n}x_{2,1}) & \exp(x_{1,n}x_{2,2}) & \cdots & \exp(x_{1,n}x_{2,n})
\end{vmatrix}
\]
is positive.

If $x_{1,n} \ge 0$ and $x_{2,n} \ge 0$, then the strict positivity of the determinant follows from \eqref{schurhyperdetsum} with~${m=1}$.

For $x_{2,n} < 0 < x_{1,n}$, we multiply the $j$th row of the determinant by $\exp(-x_{1,j}x_{2,n})$, $j=1,\dots,n$. This converts the determinant to
\[
\begin{vmatrix}
\exp(x_{1,1}(x_{2,1}-x_{2,n})) & \exp(x_{1,1}(x_{2,2}-x_{2,n})) & \cdots & \exp(x_{1,1}(x_{2,n}-x_{2,n})) \\
\vdots & \vdots & & \vdots \\
\exp(x_{1,n}(x_{2,1}-x_{2,n})) & \exp(x_{1,n}(x_{2,2}-x_{2,n})) & \cdots & \exp(x_{1,n}(x_{2,n}-x_{2,n}))
\end{vmatrix}.
\]
On the other hand, if $x_{1,n} < 0 < x_{2,n}$, then we multiply the $j$th column by $\exp(x_{1,n}x_{2,j})$, $j=1,\dots,n$; and if $x_{1,n} < 0$ and $x_{2,n} < 0$, then we carry out both multiplications.

The outcome of these multiplications is to transform the original determinant to a new determinant of the same type, with generic entries $\exp(y_{1,i}y_{2,j})$ where $y_{1,1} > y_{1,2} > \cdots > y_{1,n} \ge 0$ and $y_{2,1} > y_{2,2} > \cdots > y_{2,n} \ge 0$, and we have seen before that the resulting determinant is positive. Moreover, the two determinants differ in value only by positive (exponential) factors, therefore the initial determinant is positive.
\end{Example}

For nonnegative integers $p$ and $q$, denote by $\mcalD_{p,q}$ the set of all $(x_1,\dots,x_{2m}) \in \R^{2m}$ such that the generalized hypergeometric series ${}_pF_q(a_1,\dots,a_p;b_1,\dots,b_q;x_1\cdots x_{2m})$ converges. The next result is motivated by Example \ref{hgfhyperdet}, where the hyperdeterminant was constructed using this generalized hypergeometric series, and it extends the previous example involving the exponential function.

\begin{Proposition}
\label{prop_hgfexample}
Let $a_1,\dots,a_p \ge 0$ and $b_1,\dots,b_q > 0$. The kernel $K\colon [0,\infty)^{2m} \to \R$ such that
\[
K(x_1,\dots,x_{2m}) = {}_pF_q(a_1,\dots,a_p;b_1,\dots,b_q;x_1\cdots x_{2m})
\]
is HTP$_\infty$ on the region $\mcalD_{p,q} \cap [0,\infty)^{2m}$.
\end{Proposition}

The proof of Proposition~\ref{prop_hgfexample} follows from the hyperdeterminantal summation formula \eqref{eq_hgfhyperdet_exp} given in Example \ref{hgfhyperdet}.

\begin{Example}
In Proposition~\ref{prop_hgfexample}, suppose that $(p,q)=(1,0)$ and $a_1 \equiv a \ge 0$. It is well known that ${}_1F_0(a;x) = (1-x)^{-a}$, $|x| < 1$, and therefore $K(x_1,\dots,x_{2m}) = (1 - x_1 \cdots x_{2m})^{-a}$, $|x_1\cdots x_{2m}| < 1$. Hence the kernel $K$ is HTP$_\infty$ on the region $\{(x_1,\dots,x_{2m})\mid x_1\cdots x_{2m} < 1,\, x_1 \ge 0,\dots,x_{2m} \ge 0\}$.
\end{Example}

\subsection{The basic composition formula}

In the classical setting, examples of kernels that are hyperdeterminantal totally positive often are constructed using the {\it basic composition formula} \cite{Karlin}: If $K_1, K_2 \colon \R \to \R$ are TP$_d$ kernels, then so is the kernel
\[
K(x_{1,1},x_{2,1}) = \int_\R K_1(x_{1,1},t) K_2(t,x_{2,1}) \dd\mu(t),
\]
where $\mu$ is a positive Borel measure such that the integral converges absolutely. In the hyperdeterminantal setting, the following generalization of the basic composition formula follows from the hyperdeterminantal Binet--Cauchy formula in Proposition~\ref{prop_binet_cauchy}.

\begin{Proposition}\label{prop_bcf}
Let $(X,\mu)$ be a totally ordered sigma-finite measure space, and let $\{\phi_{k,r}\mid 1 \le k \le 2m,\, 1 \le r \le d\}$ be a collection of real-valued functions on $X$. Suppose for each $k=1,\dots,2m$ that the kernel $K_k(t,r) = \phi_{k,r}(t)$, $(t,r) \in X \times \{1,\dots,d\}$, is TP$_d$. Define the multidimensional array
\[
A(r_1,\dots,r_{2m}) = \int_X \prod_{k=1}^{2m} \phi_{k,r_k}(x) \dd\mu(x),
\]
where $1 \le r_1,\dots,r_{2m} \le d$, and the integral is assumed to converge absolutely. Then the array~${(A(r_1,\dots,r_{2m}))_{1 \le r_1,\dots,r_{2m} \le d}}$ is HTP$_d$.
\end{Proposition}

\begin{proof}
For any $n$ such that $1 \le n \le d$, consider the hyperdeterminantal minor
\[
\Det(A(r_1,\dots,r_{2m}))_{1 \le r_1,\dots,r_{2m} \le n}.
\]
By applying \eqref{binetcauchy2} in Proposition~\ref{prop_binet_cauchy}, we can represent this minor as an integral whose integrand is nonnegative, hence the minor is nonnegative. Since $n$ is arbitrarily chosen, then it follows that the array $(A(r_1,\dots,r_{2m}))_{1 \le r_1,\dots,r_{2m} \le d}$ is HTP$_d$.
\end{proof}

As a consequence of this result, numerous HTP kernels can be constructed. For example, by choosing $\phi_{k,r}(t) = t^{k+\lambda_r}$, where $\lambda_1,\dots,\lambda_r$ are the parts of a~partition~$\lambda$, then the resulting minor $\Det(A(r_1,\dots,r_{2m}))_{1 \le r_1,\dots,r_{2m} \le n}$ is a hyperdeterminantal generalization of the Vandermonde determinant, and we obtain positivity results similar to those given in \cite{wanglixu}. If we choose $\phi_{k,r}(t)$ as a classical generalized hypergeometric series and if the measure $\mu$ is chosen suitably, then the entries $A(r_1,\dots,r_{2m})$, $1 \le r_1,\dots,r_{2m} \le n$, of the resulting array can be made to contain a~wide range of the classical special functions, such as the Bessel or the Gaussian hypergeometric series. By choosing $\phi_{k,r}(t)$ in terms of the bivariate normal distributions with negative correlations~\cite{karlinrinott80}, the entries of the resulting array can be made to include numerous multivariate probability distributions.

\section{Concluding remarks}
\label{sec_conclusions}

The results presented in this article raise many open problems in a variety of areas. For the purposes of making applications in statistics and probability, as in \cite{karlinrinott80}, we are especially interested in hyperdeterminantal generalizations of the FKG inequality. Here, even the most basic questions are still open, e.g., the determination of conditions under which the familiar multivariate probability distributions, such as the multivariate normal distributions, are HTP of any order.

A second problem arises from the definition of the hyperdeterminant. Suppose that the symmetric group $\frS_n$ in the sums that define $\Det(K(x_{1,r_1},x_{2,r_2},\dots,x_{2m,r_{2m}}))$ in \eqref{htphyperdet} were to be replaced by an arbitrary finite reflection group. Following the approach in \cite{grossrichards95}, corresponding generalizations of the Binet--Cauchy formula can be obtained and examples of hyperdeterminantal totally positive kernels can be constructed. It would be a fundamental achievement to derive analogs of the FKG inequality in this setting.

\subsection*{Acknowledgements}
The results in this article were first presented in \cite{Johnson_Richards} at the {\it Second CREST-SBM International Conference}, ``Harmony of Gr\"obner Bases and the Modern Industrial Society'', held June 28 -- July 2, 2010, in Osaka, Japan (see \cite{Hibi}). We express our gratitude to the organizers of the meeting for the opportunity to have presented our results there.

We are also grateful to SIGMA's referees for their meticulous reading of the article and very helpful comments.

\pdfbookmark[1]{References}{ref}
\LastPageEnding

\end{document}